\documentclass[twocolumn]{autart}    
\usepackage{graphicx}
\usepackage[nodisplayskipstretch]{setspace}
\usepackage{amsfonts,latexsym,
amssymb,
amsmath,url,stackengine}
\usepackage{booktabs}
\usepackage{amsmath}
\usepackage{picins}
\usepackage{color}
\usepackage{empheq}

\catcode`\@=11
\def\downparenfill{$\m@th\braceld\leaders\vrule\hfill\bracerd$}
\def\downparenfill{$\m@th\braceld\leaders\vrule\hfill\bracerd$}
\def\overparen#1{\mskip 2mu\mathop{\vbox{\ialign{##\crcr\crcr \noalign{\kern0.4ex}
\downparenfill\crcr\noalign{\kern0.4ex\nointerlineskip}
$\hfil\displaystyle{#1}\hfil$\crcr}}}\limits\mskip 2mu} % <-- ajout de "\mskip 2mu" le 10 mars 2020
\catcode`\@=12

\usepackage{rgsMacros}
\usepackage{subcaption}
\usepackage{bm}
\let\labelindent\relax
\usepackage{enumitem}
\usepackage{balance}
\usepackage{comment}
\definecolor{olivegreen}{rgb}{0,0.5,0}

\newtheorem{thmm}{Theorem}
\newtheorem{propo}{Proposition}
\newtheorem{lemm}{Lemma}
\newtheorem{exe}{Example}
\newtheorem{corol}{Corollary}
\newtheorem{ass}{Assumption}
\newtheorem{prope}{Property}
\newtheorem{defin}{Definition}
\newtheorem{remm}{Remark}

\newenvironment{lemma}{\begin{lemm}}{\hfill \end{lemm}}

\newenvironment{property}{\begin{prope}}{\hfill \end{prope}}

\newenvironment{remark}{\begin{remm} \rm}{\hfill \end{remm}}

\newenvironment{assumption}{\begin{ass}}{\end{ass}}

\newenvironment{theorem}{\begin{thmm}}{\hfill $\square$ \end{thmm}}

\newenvironment{proposition}{\begin{propo}}{\hfill $\square$ \end{propo}}

\newenvironment{definition}{\begin{defin}}{\hfill \end{defin}}

\newenvironment{proof}{{\it Proof. }}{\hfill $\blacksquare$ }

\newtheorem{dwellt}{Condition}

\newtheorem{feedback}{Feedback}

\newtheorem{adapt_alg}{Algorithm}

\newif\ifitsdraft

\makeatletter
\def\@citex[#1]#2{%
  \let\@citea\@empty
  \@cite{\@for\@citeb:=#2\do
    {\@citea\def\@citea{,\penalty\@m\hskip.1em}%
     \edef\@citeb{\expandafter\@firstofone\@citeb\@empty}%
     \if@filesw\immediate\write\@auxout{\string\citation{\@citeb}}\fi
     \@ifundefined{b@\@citeb}{\mbox{\reset@font\bfseries ?}%
        \G@refundefinedtrue
        \@latex@warning
          {Citation `\@citeb' on page \thepage \space undefined}}%
        {\hbox{\csname b@\@citeb\endcsname}}}}{#1}}
\makeatother
 
\begin{document}

\begin{frontmatter}
\title{Stabilization of Quasilinear Parabolic Equations by Cubic Feedback at Boundary with Estimated Region of Attraction} 

\author[UGA]{M. C. Belhadjoudja}\ead{mohamed.belhadjoudja@gipsa-lab.fr}, 
\author[UGA]{M. Maghenem}\ead{mohamed.maghenem@gipsa-lab.fr},
\author[UGA,Dalhousie]{E. Witrant}\ead{emmanuel.witrant@gipsa-lab.fr},
\author[UCSD]{M. Krstic}\ead{mkrstic@ucsd.edu} 

\address[UGA]{Universit\'e Grenoble
Alpes, CNRS, Grenoble-INP, GIPSA-lab, F-38000, Grenoble, France}
\address[Dalhousie]{Departement of Mechanical
Engineering, Dalhousie University, Halifax B3H 4R2, Nova Scotia, Canada}
\address[UCSD]{Department of Mechanical and Aerospace Engineering, University of California San Diego, 92093 San Diego, USA}


\begin{abstract}
For quasilinear parabolic partial differential equations (PDEs) that exhibit finite-time blow up in open loop, i.e., under null boundary conditions, we provide an estimate of the region of attraction under cubic feedback laws applied at the boundary, using boundary measurements. We guarantee: 1- $L^2$ and $H^1$ exponential stability of the origin with an estimate of the region of attraction. 2- Convergence of the $H^2$ and the $\mathcal{C}^1$ norms of the solutions to zero. 3- Existence and uniqueness of complete \textit{classical} solutions. 4- Positivity of the solutions starting from positive initial conditions. Unlike existing approaches, our framework handles nonlinear state-dependent diffusion, convection, and (destabilizing) reaction. The cubic terms are used to enlarge our estimate of the region of attraction. The size of the region of attraction is shown, in many cases, to grow unboundedly as diffusion increases. Finally, our controllers can be implemented as Neumann, Dirichlet, or mixed-type boundary conditions.
\end{abstract}

\end{frontmatter}

\setlength{\abovedisplayskip}{6pt}
\setlength{\belowdisplayskip}{6pt}
\setlength{\abovedisplayshortskip}{3pt}
\setlength{\belowdisplayshortskip}{3pt}

\section{Introduction} \label{Introduction}
Quasilinear parabolic PDEs exhibiting 
finite-time blow-up phenomena have been studied extensively since the 1940s. Initially motivated by adiabatic explosion \cite{combustion2}, these equations have later gained interest in light of the seminal work \cite{fujita1}. Nowadays, these equations are fundamental in many applications, including chemical kinetics \cite{chemical1} and plasma physics \cite{plasma}; see \cite{galaktionov} for further details.

Finite-time blow up refers to the max norm of the state over the spatial domain growing unboundedly in finite time. This phenomena has drawn the attention of the control community, with the ultimate goal of designing controllers to prevent it. However, these efforts were curtailed since \cite[Theorem 1.1]{lack2}, where it is demonstrated that, for simple instances of the heat equation, it is not possible to achieve global stability using neither boundary nor in-domain controls. In particular, finite-time blow up is doomed to occur for large initial conditions, regardless of the choice of control inputs. As a consequence, for general classes of quasilinear parabolic equations, the best achievable stability results come in the form of an estimate of the region of attraction. Yet, estimating the region of attraction is of significant interest as, under null boundary inputs, finite-time blow up can still occur for arbitrarily small initial conditions \cite[Page 11]{book_quasi}.

Although the control of parabolic PDEs has been intensively studied (see e.g. \cite{par1,backstepping,par8,par9,par11,par12,par13,coron_steadystate}), boundary  stabilization of quasilinear parabolic PDEs, which may otherwise exhibit finite-time blow up, remains unexplored. To our knowledge, the only studies addressing this problem pertain to the specific subclass of semilinear parabolic equations with \textit{Volterra-type} nonlinearity, and employ nonlinear backstepping \cite{volterra1,volterra2}. This method, however, cannot handle nonlinear state-dependent diffusion or convection, besides the incapacity to assure positivity of solutions. That being said, it is important to mention \cite{shock}, where backstepping is applied to regulate shock-like equilibria in Burgers’ equation with quadratic convection. The error PDE exhibits in open loop finite-time blow up, but the approach relies on the Cole-Hopf transform—a change of variables specific to Burgers’ equation that linearizes quadratic convection. We mention also \cite{par10}, which--although not considering PDEs exhibiting finite-time blow up--is the only work achieving global stabilization for parabolic equations with polynomial reaction, but under the assumption that the nonlinearity is stabilizing when the state is large. Moreover, the method in \cite{par10} cannot handle state-dependent diffusion and nonlinear convection, and does not ensure positivity of solutions. Such limited results underscore the need for a more general control framework that handles nonlinear state-dependent diffusion, convection, and--destabilizing--reaction.

In this paper, we propose boundary controllers, in the form of cubic polynomials of the state at the boundaries, that achieve: 1- Exponential stability of the origin in the $L^2$ and $H^1$ norms with an estimate of the region of attraction. 2- Exponential convergence of the $H^2$ and $\mathcal{C}^1$ norms of the state to zero. 3- Existence and uniqueness of complete classical solutions. 4- Positivity of solutions, i.e., the solution remains nonnegative if the initial condition is nonnegative. 
Additionally, our boundary controllers can be written as Neumann-type, Dirichlet-type, or Dirichlet-Neumann-type controllers. The passage from one controller to another is achieved through solving cubic equations using Cardano's root formula. We further show that, in many cases, the region of attraction---along with the decay rate---grows unboundedly as the diffusion coefficient increases. To argue the importance of such results, we recall that, if the boundary control inputs are set to zero, finite-time blow up occurs even under 
arbitrarily-small initial conditions and arbitrarily-large diffusion coefficients \cite{book_quasi,conf}. 
The philosophy of our design is inspired by \cite{burgers}, the first work to ever use cubic polynomials, in the boundary states, as boundary controllers. However, \cite{burgers} considers the viscous Burgers' equation, which does not exhibit finite-time blow up. A major challenge in our study, absent in \cite{burgers} and similar  works (see e.g.  \cite{balogh2,kdv2,dec,ks}) concerns handling the reaction term responsible of finite-time blow up in open loop. 
From the technical point of view, this term actually prevents us from analyzing the state's $L^2$ and max norms separately, which is a key ingredient in \cite{burgers}. 

We shall also mention that, in \cite{ACC24,aut_accepted}, we extended the controllers in \cite{burgers,ACC_Moh,quadratic} to semilinear parabolic PDEs, using controllers that involve some state norms. This allowed us to fully compensate the effect of the reaction term in \cite{ACC24}, and of the unknown \textit{anti-diffusion} parameter in the Kuramoto-Sivashinsky equation in \cite{aut_accepted}, and to deduce the decay of the solutions' $L^2$ norm towards zero. However, because the $H^1$ norm of the state was not analyzed, we could not establish well-posedness. The present work addresses this limitation for second-order quasilinear parabolic PDEs. By designing controllers that rely only on boundary measurements—eschewing the state norms in the controllers—we are able to analyze the $H^1$ norm of the state and, consequently, establish well-posedness. The price we pay for avoiding the use of state norms in the controllers is that, unlike in \cite{ACC24} where diffusion played no role, in the current work, the size of the region of attraction depends critically on diffusion: it grows as diffusion increases and shrinks as diffusion decreases. In this sense, our approach is \textit{diffusion-enabled}, in contrast with the \textit{convection-enabled} strategy of \cite{ACC24}, where the convection term led, as in \cite{ACC_Moh,aut_accepted}, to a cubic polynomial in the Dirichlet boundary input that was used to compensate for all potentially destabilizing effects, using, in particular, state norms. This reliance on diffusion makes our approach specific to parabolic PDEs.

The remainder of this paper is organized as follows. The problem statement is given in Section \ref{problem statement}. The main result is in Section \ref{main results}. The paper is ended with a conclusion and some research perspectives.

Unlike a preliminary conference version \cite{conf}, this paper provides proofs and explanations, and presents new stability and convergence results in higher norms---cf. the forthcoming Theorem \ref{thm1}.

\textit{Notation.} Given $v : [0,1] \to \mathbb{R}$, we let $|v|_{\infty} := \max_{x \in [0,1]}|v(x)|$. Similarly, given $T \in (0,+\infty]$,  
$u : [0,1] \times [0,T) \to \mathbb{R}$, and $t \in [0,T)$, 
 we let $|u(\cdot,t)|_{\infty}:=\max_{x \in [0,1]}|u(x,t)|$. We also write $|u|_{\infty}$ to mean the map $t \mapsto |u(\cdot,t)|_{\infty}$.
Furthermore, we denote by $u_x$ the partial derivative of $u$ with respect to $x$, $u_{xx}$ the second partial derivative of $u$ with respect to $x$, and $u_t$ the partial derivative of $u$ with respect to $t$. Moreover, for each $t \in [0,T)$, we let $|u(\cdot,t)|_{L^2} := (\int_{0}^{1}u(x,t)^2dx)^{1/2}$, $|u(\cdot,t)|_{H^1} := (|u(\cdot,t)|_{L^2}^2+|u_x(\cdot,t)|_{L^2}^2)^{1/2}$, $|u(\cdot,t)|_{H^2} := (|u(\cdot,t)|_{H^1}^2+|u_{xx}(\cdot,t)|_{L^2}^2)^{1/2}$, and $|u(\cdot,t)|_{\mathcal{C}^1} := |u(\cdot,t)|_{\infty}+|u_x(\cdot,t)|_{\infty}$. We also write $|u|_{L^2}$ and $|u|_{H^1}$ to mean the maps $t\mapsto |u(\cdot,t)|_{L^2}$ and $t \mapsto |u(\cdot,t)|_{H^1}$, respectively. Given $\beta \in (0,1)$, we denote by $\mathcal{H}^{2+\beta}[0,1]$ 
the space of twice continuously-differentiable maps $v: [0,1] \rightarrow \mathbb{R}$ such that $v''$, the second-order derivative of $v$, is Hölder continuous with exponent $\beta$, i.e., there exists $C>0$ such that 
$ |v''(x_1)-v''(x_2)| \leq C |x_1-x_2|^{\beta}$ for all $x_1,x_2 \in [0,1]$. 
In particular, we denote by $\mathcal{H}^{2+\beta,1+\beta/2}([0,1] \times [0,T))$, the space of maps $u : [0,1] \times [0,T) \to \mathbb{R}$ such that $u_{xx}$ and $u_t$ are Hölder continuous in $x$ with exponent $\beta$ uniformly in $t$ and Hölder continuous in $t$ with exponent $\beta/2$ uniformly in $x$.  
We denote by $\mathcal{H}_{loc}^{2+\beta,1+\beta/2}([0,1]\times [0,T))$ the space of maps $u :[0,1] \times [0,T) \to \mathbb{R}$ such that the restriction of $u$ to any bounded subset $[0,1] \times L \subset [0,1] \times [0,T)$ belongs to $\mathcal{H}^{2+\beta,1+\beta/2}([0,1] \times L)$.

\section{Problem formulation} \label{problem statement}
\subsection{Class of systems}\label{open_loop}
Consider the quasilinear parabolic PDE 
\begin{empheq}[left=\Sigma:\ \left\{,right=\right.]{align}
&u_t   = \varepsilon(u)u_{xx} + \sum_{i=1}^{n}\gamma_i u^i u_x + u^p, \label{one}\\
&u_x(0)= v_0(u(0)), \ \ u_x(1) = v_1(u(1)), \label{two}\\
&u(x,0)= u_0(x)  \in  \mathcal{H}^{2+\beta}[0,1], ~~ \beta \in (0,1), \label{three}
\end{empheq}
where $x\in [0,1]$ denotes the space variable, $t\geq 0$ denotes  the time variable, and $u \in \mathbb{R}$ denotes the state variable. 

The term $\varepsilon(u)u_{xx}$ is referred to as \textit{diffusion}, with $\varepsilon>0$ being the  
state-dependent diffusion coefficient, verifying the following assumption. 

\begin{assumption}\label{ass1}
The map $u \mapsto \varepsilon(u)$ is differentiable, and there exist $\underline{\varepsilon}>0$ and a continuous non-decreasing map $\bar{\varepsilon} : \mathbb{R}_{\geq 0}\to \mathbb{R}_{\geq 0}$ such that
$\inf_{u\in \mathbb{R}} \{\varepsilon(u)\} \geq \underline{\varepsilon}$ and 
\begin{align*}
|\varepsilon'(u)|  \leq \bar{\varepsilon}(s) \qquad  \forall u \in [-s,s], ~~ \forall s\geq 0.
\end{align*}
\end{assumption}

The term $\sum_{i=1}^{n}\gamma_i u^i u_x$, for $n \in \{1,2,...\}$ and $\gamma_i \in \mathbb{R}$ for all $i \in \{1,2,...,n\}$, represents the \textit{convection} with 
$\{\gamma_i\}_{i=1}^{n}$ being the convection coefficients. The  term $u^p$, $p \in (1,+\infty)$, is the \textit{reaction}, the cause of finite-time blow up. Furthermore, the boundary conditions in \eqref{two}, of Neumann type, involve the inputs $v_1, v_o: \mathbb{R}\to \mathbb{R}$ to be designed.
Finally, the choice of the initial condition in \eqref{three} is necessary for the existence of classical solutions \cite{ladyzen}.

\begin{definition}\label{def1}
A classical solution to $\Sigma$ defined on $[0,T)$, for some $T\in (0,+\infty]$, is any map $u \in \mathcal{H}^{2+\beta,1+\beta/2}_{loc}([0,1]\times [0,T))$ that verifies \eqref{one} for all $(x,t)\in (0,1)\times (0,T)$, \eqref{two} for all $t\in [0,T)$, and \eqref{three} for all $x\in [0,1]$. 
\end{definition}

A solution $u$ is \textit{complete} if it is defined on $[0,+\infty)$. It is said to blow up in finite time if there exists $T_m \in (0,+\infty)$ such that $\lim_{t\rightarrow T^{-}_m} |u(\cdot,t)|_{\infty} = +\infty$.

\begin{remark}
According to Definition \ref{def1}, a necessary condition for the existence of solutions is
\begin{align}
u_o'(0) = v_o(u_o(0)), \  \  \ u_o'(1) = v_1(u_o(1)), \label{compatibility}
\end{align}
known as the \textit{$0$th-order compatibility} condition \cite{ladyzen}.
\end{remark}

Equation $\Sigma$ constitutes a fairly-general prototype to study parabolic PDEs exhibiting finite-time blow up. 
Indeed, for the specific  quasilinear heat equation $u_t = \varepsilon(u) u_{xx} + u^2$, with $\varepsilon \in \mathcal{C}^2$ and lower bounded by a positive constant, finite-time blow up occurs under $(v_o,v_1):=0$ whenever $\int_{0}^{1} u_o(x)dx > 0$. Moreover, the blow-up time is bounded by a constant independent of $\varepsilon$; see \cite[Page 11]{book_quasi} and \cite[Section II.A]{conf}. Finally, it is essential to mention that $\Sigma$ is not globally stabilizable in general. In particular, for $\varepsilon$ constant, $\gamma_i:=0$ for all $i\in \{1,2,...,n\}$, and $p:=2$, there are initial conditions from which, given any two maps $v_o$ and $v_1$, the solutions (when they exist) necessarily blow up in finite time \cite{volterra1}.

\subsection{Control objective}

Our goal is to design $(v_o,v_1)$ to establish the following three properties for the resulting closed-loop system.

\begin{property}[Stability and convergence]\label{prop1}
There exist $\lambda_o,\lambda_1, \omega >0$ and $\mu \geq 0$ such that, provided that
\begin{align}
\kappa_o(u_o) & := \bigg[|u_o|^2_{H^1} + \lambda_1u_o(1)^2+ \frac{\mu}{2}u_o(1)^4\nonumber \\
&~\quad + \lambda_o u_o(0)^2+ \frac{\mu}{2}u_o(0)^4\bigg]^{1/2} \leq \sqrt{2\omega}, \label{init_func}
\end{align}
 there exist $\sigma(\kappa_o)>0$ and 
$\bar{\zeta} : \mathbb{R}_{\geq 0} \times \mathbb{R}_{\geq 0} \rightarrow \mathbb{R}_{\geq 0}$ continuous,  nondecreasing in both arguments, and $\bar{\zeta}\left(0,0\right) = 0$,
such that every complete  solution satisfies, for all $t \geq 0$,
\begin{align}
|u(\cdot,t)|^2_{L^2} & \leq |u_o|^2_{L^2}\exp^{-\sigma t},  \label{decay_L2} \\
|u(\cdot,t)|_{\infty}^2 & \leq 2\kappa_o |u_o|_{L^2}  \exp^{-(\sigma/2)t} + |u_o|^2_{L^2} \exp^{-\sigma t}, \label{decay_max} \\
|u_x(\cdot,t)|^2_{L^2}
& \leq \bar{\zeta}(\kappa_o, |u_o|_{L^2})  \exp^{-(\sigma/3)t}  - \lambda_1 u(1,t)^2 \nonumber \\ & - \lambda_ou(0,t)^2 
 - \frac{\mu}{2} u(1,t)^4 - \frac{\mu}{2}u(0,t)^4. \label{decay_H1} 
 \end{align}
Moreover, 
\begin{equation}
\label{max_ux}
\begin{aligned}
& \lim_{t\to +\infty} \exp^{(\sigma/4)t} \times \\ &
\max\left\{ |u_{xx}(\cdot,t)|^2_{L^2}, |u_t(\cdot,t)|^2_{L^2}, |u_x(\cdot,t)|_{\infty}^2 \right\} = 0. 
 \end{aligned}
 \end{equation}
\end{property}

\begin{property}[Well-posedness]\label{prop2} 
Under the compatibility condition \eqref{compatibility},
there exists a unique solution that is complete, and any other solution must be a restriction of the complete solution.
\end{property}

\begin{property}[Positivity]\label{prop3}
If $u_o(x)\geq 0$ for all $x \in [0,1]$ and $u$ is a complete solution to $\Sigma$, then $u(x,t) \geq 0$ for all $(x,t)\in [0,1]\times [0,+\infty)$.
\end{property}

Additionally, $(v_o,v_1)$ must verify the following property. 

\begin{property}[Invertibility] \label{prop4}
The maps $v_o$ and $v_1$ are invertible, i.e., there exist $d_o$, $d_1 : \mathbb{R} \rightarrow \mathbb{R}$ such that, given any solution $u$ to $\Sigma$,
\begin{equation*}
\begin{aligned}
u_x(0) = v_o(u(0)) & \Longleftrightarrow u(0) = d_o(u_x(0)), \\
 \ \ u_x(1) = v_1(u(1)) & \Longleftrightarrow u(1) = d_1(u_x(1)).
\end{aligned}
\end{equation*}
\end{property}

\begin{remark}
Using Agmon's inequality in the Appendix, we obtain
\begin{align*}
 \kappa_o  \leq 2\sqrt{\mu}|u_o|_{H^1}^2  +\sqrt{1+4\max\{\lambda_o,\lambda_1\}}|u_o|_{H^1}.
\end{align*}
Hence, provided that the $H^1$ norm of $u_o$ is sufficiently small, we can achieve \eqref{init_func}. Furthermore, Property \ref{prop1} entails both $L^2$ and $H^1$ exponential stability of the origin with an estimate of the region of attraction. It is interesting to note that the limit in \eqref{max_ux} holds regardless  the size of $|u_{o}''|_{L^2}$, $|u_t(\cdot,0)|_{L^2}$, and $|u_o'|_{\infty}$. 
\end{remark}

\begin{remark}
 Property \ref{prop3} can be crucial in applications where negative values of the state are meaningless. This is the case in  autocatalytic processes \cite{chemical1}, where the state represents a chemical concentration. 
\end{remark}

 \begin{remark}
 The invertibility of $v_o$ and $v_1$ in Property \ref{prop4} makes our design adaptable to different applications. For instance, in fluid systems, the Dirichlet-type actuation $(u(0), u(1)) = \left( d_o(u_x(0)),  d_1(u_x(1)) \right)$ is often more feasible than Neumann-type actuation \cite{dirichlet2}. However, in thermal systems, Neumann-type actuation is typically more practical than Dirichlet-type actuation \cite{balogh2}; see also \cite[Equations (3.12) and (3.13)]{burgers}, \cite[Remark 2.3]{dec}, and \cite[Equation (2.25)]{ks}.
\end{remark}

\section{Main results} \label{main results}

To state our main result, we first  design $(v_o,v_1)$ as
\begin{align}
v_o(u(0)) & := \lambda_ou(0)+\mu u(0)^3, \label{v2} \\
v_1(u(1)) & := -\lambda_1u(1)-\mu u(1)^3, \label{v1}
\end{align}
where $\mu := \frac{1}{\underline{\varepsilon}}\bigg( \frac{|\gamma_1|}{2} +\sum_{i=2}^{n}\frac{|\gamma_i | M^{i-2}}{i+2}\bigg)$ and 
\begin{align*}
\lambda_1 & := \frac{2(m_1+k_1)+|\gamma_1|}{2\underline{\varepsilon}}, ~
\lambda_o := \frac{2(m_o+k_o)+|\gamma_1|}{2\underline{\varepsilon}}.
\end{align*}
Here, $M, m_1, m_o>0$ and $k_1,k_o\geq 0$ are free   parameters. 

\begin{remark}
 When $\gamma_i = 0$ for all $i\in \{1,2,...,n\}$,  we recover the so-called Robin's boundary conditions $u_x(1) = - \lambda_1 u(1)$ and $u_x(0)=\lambda_o u(0)$.
\end{remark}

\begin{theorem}\label{thm1}
Consider system $\Sigma$ such that Assumption \ref{ass1} holds. Let $(v_o,v_1)$ as in  \eqref{v2}-\eqref{v1}, respectively. 
Then, 
\begin{itemize}
\item Property \ref{prop1} holds for any $\omega>0$  verifying 
\end{itemize}
\begin{align}
m_1+m_o & > 2(6\omega)^{(p-1)/2}+(6\omega)^{p-1}/\underline{\varepsilon}, \label{omega_1} \\
\underline{\varepsilon}-2(m_1+m_o)& > \sqrt{6\omega}\bar{\varepsilon}(\sqrt{6\omega}) +\frac{n}{2\underline{\varepsilon}}\bigg(\sum_{i=1}^{n}\gamma_i^2(6\omega)^i\bigg), \label{omega2}
\end{align}
\begin{align}
\sum_{i=2}^{n}\frac{|\gamma_i|}{i+2}M^{i-2} & > \sum_{i=2}^{n}\frac{|\gamma_i|}{i+2}(6\omega)^{(i-2)/2} \nonumber \\
&~\quad \text{if $\exists i\in \{2,3,...,n\}$ \text{s.t.}  $\gamma_i\neq 0$,} \label{omega3}
\end{align}
and for 
\begin{align}
&\sigma := m_1+m_o-2\left(\sqrt{3}\kappa_o\right)^{(p-1)} > 0, \label{sigma} \\
&\bar{\zeta} := \kappa_o^2+ 4(\lambda_1+\lambda_o)\kappa_o|u_o|_{L^2}+\frac{\mu}{5}|u_o|_{L^2}^4+\frac{8\mu \kappa_o}{7}|u_o|_{L^2}^{3} + \nonumber \\
&\left[\frac{\lambda_1+\lambda_o+4\mu \kappa_o^2}{2} + \frac{3^p\kappa_o^{2p-2}}{2\underline{\varepsilon}\big(m_1+m_o-2(6\omega)^{(p-1)/2}\big)} \right] |u_o|_{L^2}^2
\nonumber \\ &
 +   \frac{(3/4) \bigg[\bigg(\sum_{i=1}^{n} \frac{n 3^i\gamma_i^2 \kappa_o^{2i}}{\underline{\varepsilon}} \bigg) +\frac{m_1+m_o}{3}\bigg]}{\underline{\varepsilon} - \sqrt{6\omega} \bar{\varepsilon}\left(\sqrt{6\omega}\right)-2(m_o+m_1)} |u_o|_{L^2}^2.
\end{align}
\begin{itemize} 
\item Property \ref{prop2} holds when $\varepsilon \in \mathcal{C}^2$ and $\kappa_o(u_o) \leq \sqrt{2\omega}$ with $\omega$ verifying \eqref{omega_1}-\eqref{omega3}.

\item Property \ref{prop3} holds when  $\varepsilon' \in \mathcal{H}_{loc}(\mathbb{R})$.
\end{itemize}
\begin{itemize}
\item Property \ref{prop4} holds, when $\mu \neq 0$, with 
\begin{equation}
\label{cubic1}
\begin{aligned} 
\hspace{-0.7cm} d_{l}(y) & := \sqrt[3]{ \frac{(-1)^{l} y}{2\lambda_l}+\sqrt{\frac{y^2}{4\lambda_l^2} + \frac{\lambda_l^3}{27\mu^3}}}  \\
\hspace{-0.7cm} & +\sqrt[3]{ \frac{(-1)^{l} y}{2\lambda_l}-\sqrt{\frac{y^2}{4\lambda_l^2} + \frac{\lambda_l^3}{27\mu^3}}} ~~~~ \forall l \in \{0,1\}, 
\end{aligned}
\end{equation}
and, when $\mu = 0$, with 
$ d_{l}(y) :=  (-1)^{l} y/\lambda_l$ for all 
$l \in \{0,1\}$. 
\end{itemize}
\end{theorem}

\begin{remark}
Since $\bar{\varepsilon}$ is continuous and nondecreasing, then the inequalities \eqref{omega_1}-\eqref{omega3} would hold for $\omega$ sufficiently small, provided that $m_o$ and $m_1$ are chosen such that $\underline{\varepsilon}  > 2(m_o+m_1)$. Furthermore, for $\underline{\varepsilon}  > 2(m_o+m_1)$, $p > 1$, $M:=M(\underline{\varepsilon})$ with $\lim_{\underline{\varepsilon}\to +\infty}M(\underline{\varepsilon})=+\infty$, and $\bar{\varepsilon}$ uniformly bounded, it holds that  $\lim_{\underline{\varepsilon}\to +\infty} \Omega(\underline{\varepsilon})=+\infty$, where 
$\Omega (\underline{\varepsilon}) := \sup \{\omega : \text{\eqref{omega_1}-\eqref{omega3} hold} \}$ represents a lower bound on the  size of the region of attraction. Finally, according to \eqref{sigma}, the decay rate $\sigma$ verifies $\sigma \to +\infty$ as $\underline{\varepsilon} \to +\infty$. However, it is important to note that while $\Omega(\underline{\varepsilon}) > 0$ for any $\underline{\varepsilon} > 0$, we have $\lim_{\underline{\varepsilon} \to 0} \Omega(\underline{\varepsilon}) = 0$. This means that although a high diffusion coefficient is not required, the region of attraction becomes smaller as the diffusion coefficient decreases.

The condition \eqref{omega3} specifies how large the control gain $M$ must be chosen to offset the effect of convection; notably, for any fixed value of $(\{\gamma_i\}_{i=2}^{n},\omega)$, \eqref{omega3} is satisfied for sufficiently large $M$. In contrast, condition \eqref{omega2} indicates how the convection coefficients constrain the size of the region of attraction. For a fixed value of $(\underline{\varepsilon},m_1,m_o,p,n)$, and assuming $\bar{\varepsilon}$ is independent of $\{\gamma_i\}_{i=1}^{n}$ and $\underline{\varepsilon}  > 2(m_o+m_1)$, $\Omega$ shrinks to zero as $\max \{\{\gamma_i^2\}_{i=1}^{n}\} \to +\infty$, and increases as $\max\{\{\gamma_i^2\}_{i=1}^{n}\}$ decays.

The influence of $p$ on $\Omega$ can be deduced from the condition \eqref{omega_1}. In particular, for a fixed value of $(\underline{\varepsilon},m_o,m_1)$ and for large enough $p$, we must have $\Omega \leq 1/6$.

Finally, while \eqref{omega_1}-\eqref{omega3} restrict, as discussed above, the choice of $M$, $m_o$, and $m_1$, the other gains $k_o,k_1\geq 0$ can be freely chosen. In some cases, it can be shown that, by letting $k_o>0$ or $k_1>0$, the inequality \eqref{decay_L2} on $|u|_{L^2}$ becomes strict; see Section IV and the simulation examples in the conference version \cite{conf}. 
\end{remark}
\begin{remark}
Interestingly, \eqref{omega_1} involves, after division by $\underline{\varepsilon}$, the \textit{Damköhler number} used in chemical kinetics \cite{dam}, given by  
$D_a := \omega ^{p-1}/\underline{\varepsilon}^2$, and which quantifies the relative magnitude of reaction compared to diffusion. In particular, since \eqref{omega2} requires that $\underline{\varepsilon} > 2(m_o+m_1)$, then \eqref{omega_1} imposes that $D_a<1$. In the standard (albeit informal) language of physics, the condition $D_a<1$ means that, for the initial conditions verifying $\kappa_o(u_o)\leq \sqrt{2\omega}$, the diffusion \textit{dominates} the reaction.

Similarly, \eqref{omega2} involves, after division by $\underline{\varepsilon}$, the \textit{Péclet number} \cite{peclet}, $P_e := \sum_{i=1}^{n}\left(\gamma_i(6\omega)^{i/2}/\underline{\varepsilon}\right)^2$,
which quantifies the relative magnitude of convection compared to diffusion. Condition \eqref{omega2} requires, in particular, that $P_e<1$, which in turn means that for the initial conditions verifying $\kappa_o(u_o)\leq \sqrt{2\omega}$, the diffusion dominates the convection. This result contrasts sharply with \cite{burgers}, where global stabilization of the viscous Burgers' equation, via Neumann cubic feedback, was achieved regardless of the values of the convection and diffusion coefficients. As we will see in the forthcoming proof of Theorem \ref{thm1}, the restriction on $P_e$ in our case comes from our analysis of the $H^1$ norm of $u$, which differs fundamentally from the analysis performed in \cite{burgers}, due to the presence of $u^p$.
\end{remark}

\begin{remark}
The role of the cubic terms in \eqref{v2}-\eqref{v1} is to enlarge our estimate of the region of attraction. Indeed, consider, e.g., the case where $\gamma_i = 0$ for all $i\in \{2,3,...,n\}$, and $\gamma_1\neq 0$. If we use a linear controller $(v_o(u(0)),v_1(u(1))):=(\tilde{\lambda}_ou(0),-\tilde{\lambda}_1u(1))$ instead of \eqref{v2}-\eqref{v1}, then we can show that Property \ref{prop1} holds with $\kappa_o(u_o)$ therein replaced by 
\begin{align*}
\tilde{\kappa}_o(u_o) := \left[|u_o|_{H^1}^2+\tilde{\lambda}_1u_o(1)^2+\tilde{\lambda}_ou_o(0)^2\right]^{1/2}, 
\end{align*}
$\tilde{\lambda}_l := \lambda_l + |\gamma_1|/(2\underline{\varepsilon})\tilde{M}^2$ for all $l\in \{0,1\}$, with $\tilde{M}\geq 0$ a control gain, and $\omega$ verifying \eqref{omega_1}-\eqref{omega2} and $\tilde{M}>\sqrt{6\omega}$. 
The latter condition implies that 
$$ \tilde{\lambda}_lu_o(l)^2 \geq \lambda_lu_o(l)^2 + \frac{\mu}{2}u_o(l)^4 \quad \forall l\in \{0,1\}.$$
In particular, if $\tilde{\kappa}_o(u_o)\leq \sqrt{2\omega}$, then $\kappa(u_o)\leq \sqrt{2\omega}$, while the converse is not necessarily true. The condition $\tilde{\kappa}_o(u_o)\leq \sqrt{2\omega}$ is therefore, in general, more restrictive than the condition $\kappa_o(u_o)\leq \sqrt{2\omega}$. 
\end{remark}

\begin{remark}
Existing well-posedness results for quasilinear parabolic PDEs--see the forthcoming Lemma \ref{ladyz}--cannot be directly applied to $\Sigma$. 
This is due to the superlinear reaction term $u^p$, the nonlinear convective terms $u^i u_x$, and the cubic terms $u(1)^3$ and $u(0)^3$ in \eqref{v1} and \eqref{v2}, respectively. Indeed, the latter terms prevent the first condition in Lemma \ref{ladyz} from being satisfied when applied to $\Sigma$. To overcome this issue, we will, roughly speaking, consider a \textit{cutoff version} of $\Sigma$, denoted $\bar{\Sigma}$, to which Lemma \ref{ladyz} applies. Furthermore, we exploit inequality \eqref{decay_max} to prove that, when $\kappa_o(u_o) \leq \sqrt{2\omega}$ and the compatibility condition \eqref{compatibility} holds, then the unique complete solution to $\bar{\Sigma}$ is also the unique complete solution to $\Sigma$; see Section \ref{well_pos_s}.
\end{remark}

\subsection{Further Results: One-sided boundary feedback}

We will shed light on some cases, where it is possible to guarantee Properties \ref{prop1}-\ref{prop3} by designing $v_1$ as in \eqref{v1} (resp., $v_o$ as in \eqref{v2}) while setting $v_o:=0$ (resp., $v_1:=0$).

\begin{proposition} \label{proposi1}
Consider system $\Sigma$ such that Assumption \ref{ass1} holds.  
Suppose that $\gamma_{k} := 0$ for all $k \in \{1,3, ... \}$, and $\gamma_{k}\geq 0$ (resp., $\gamma_{k}\leq 0$) for all $k\in \{2,4,...\}$. Then, by designing $v_1$ as in \eqref{v1} (resp., $v_o$ as in \eqref{v2}) while setting $v_o:=0$ (resp., $v_1:=0$), the first three bullets in Theorem \ref{thm1} remain verified.
\end{proposition}

\begin{proof}
The only difference between the proof of Theorem \ref{thm1} and the proof of Proposition \ref{prop1} is the obtained upperbound of  $\dot{V}$, where $V :=
(1/2) |u|^2_{L^2}$. 
Indeed, using \eqref{ineq_V_1}, when $\gamma_{k}:=0$ for every $k$ odd, and $\gamma_{k} \geq 0$ for every $k$ even\footnote{We handle the case where $\gamma_{k} \leq 0$ for every $k$ even in a similar way.}, we can show that 
\begin{align}
\dot{V} &\leq 2|u|_{\infty}^{p-1}V- \bigg(\underline{\varepsilon}-|u|_{\infty}\bar{\varepsilon}(|u|_{\infty})\bigg)|u_x|_{L^2}^2 \nonumber \\
&~+u(1)\varepsilon(u(1))v_1(u(1))-u(0)\varepsilon (u(0))v_o(u(0)) \nonumber \\
&~+\sum_{0\leq 2k \leq n}\frac{\gamma_{2k}}{2(k+1)}u(1)^{2(k+1)}. \nonumber 
\end{align}
Due to the absence of polynomial terms in $u(0)$ in the latter inequality ---cf. \eqref{ineq_V_1}, we can set $v_o:=0$ and  verify 
\begin{align}
\dot{V} &\leq 2|u|_{\infty}^{p-1}V- \bigg(\underline{\varepsilon}-|u|_{\infty}\bar{\varepsilon}(|u|_{\infty})\bigg)|u_x|_{L^2}^2 \nonumber \\
&+u(1)\varepsilon(u(1))v_1(u(1))+\sum_{0\leq 2k \leq n}\frac{\gamma_{2k}}{2(k+1)}u(1)^{2(k+1)}. \nonumber 
\end{align}
From here, the proof of Proposition \ref{proposi1} follows the exact same steps as the proof of Theorem \ref{thm1}.
\end{proof}

\section{Proof of Theorem \ref{thm1}}

The next four sections are dedicated to the proof of the four properties in Theorem \ref{thm1}.

\subsection{Proof of Property \ref{prop1} (Stability and convergence)}

We start considering the Lyapunov functional candidate 
$$ V(u) := \frac{1}{2}|u|_{L^2}^2. $$ 
The next lemma, proved in the Appendix, upperbounds the variations of $V$ along $\Sigma$.  
\begin{lemma}\label{lem1}
Along $\Sigma$, we have
\begin{align*}
\hspace{-0.3cm} & \dot{V}  \leq -\bigg(m_1+m_o-2|u|_{\infty}^{p-1}\bigg)V -k_1 u(1)^2 - k_o u(0)^2
\\ \hspace{-0.3cm} &
- \bigg(\underline{\varepsilon}-|u|_{\infty}\bar{\varepsilon}(|u|_{\infty}) - 2(m_o+m_1)\bigg)|u_x|_{L^2}^2+ u(1) \varepsilon(u(1)) \times
\end{align*}
\begin{align*}
& 
\bigg[ v_1(u(1))  
+ \lambda_1 u(1) + \bigg( \frac{|\gamma_1|}{2 \underline{\varepsilon}} + \sum_{i=2}^{n}\frac{|\gamma_i | |u|_{\infty}^{i-2}}{\underline{\varepsilon}(i+2)}\bigg)u(1)^3\bigg]\\
& 
+ u(0)\varepsilon (u(0)) \times 
\\ \hspace{-0.3cm} & 
\bigg[ - v_o(u(0)) +\lambda_o u(0) + 
\bigg( \frac{|\gamma_1|}{2 \underline{\varepsilon}} + \sum_{i=2}^{n}\frac{|\gamma_i | |u|_{\infty}^{i-2}}{\underline{\varepsilon} (i+2)}\bigg)u(0)^3\bigg].
\end{align*}
\end{lemma}
Under our choice of $(v_o, v_1)$ in \eqref{v2}-\eqref{v1},  respectively, the latter inequality  becomes
\begin{align}
\dot{V} & \leq -\bigg(m_1+m_o-2|u|_{\infty}^{p-1}\bigg)V -k_1 u(1)^2-k_ou(0)^2 \nonumber 
\\ & - \bigg(\underline{\varepsilon}-|u|_{\infty}\bar{\varepsilon}(|u|_{\infty})   -2(m_o+m_1)\bigg)|u_x|_{L^2}^2 \nonumber \\
&-\varepsilon(u(1))\bigg[ \sum_{i=2}^{n}\frac{|\gamma_i|}{i+2}\bigg(M^{i-2}-|u|_{\infty}^{i-2}\bigg)\bigg] u(1)^4 \nonumber \\ 
&- \varepsilon(u(0))\bigg[ \sum_{i=2}^{n}\frac{|\gamma_i|}{i+2}\bigg(M^{i-2}-|u|_{\infty}^{i-2}\bigg)\bigg] u(0)^4. \label{ineq_V_fin2}  
\end{align}
In view of \eqref{ineq_V_fin2}, to assess the behavior of $V$, we also need to analyze the behavior of $|u|_{\infty}$. Using Agmon's inequality, we conclude that 
\begin{align}
|u|_{\infty}^2 \leq 2\sqrt{2V|u_x|_{L^2}^2}+2V. \label{agmon_used}
\end{align}
Consequently, 
$|u|_{\infty}$ is bounded if 
$V$ and $|u_x|_{L^2}$ are bounded. Hence, we introduce the functional 
\begin{equation}
\label{H_def}
\begin{aligned}
H(u) & := \frac{1}{2} |u_x|_{L^2}^2 + \frac{\lambda_1 }{2} u(1)^2+ \frac{\mu}{4}u(1)^4 
\\ &
+ \frac{\lambda_o}{2}u(0)^2+ \frac{\mu}{4}u(0)^4. 
\end{aligned}
\end{equation}
The next lemma, proved in the Appendix, upperbounds the variations of $H$ along $\Sigma$.  
\begin{lemma} \label{lem2}
Along $\Sigma$, it holds that 
\begin{align}
\dot{H}\leq \frac{n}{2\underline{\varepsilon}}\bigg(\sum_{i=1}^{n}\gamma_i^2|u|_{\infty}^{2i}\bigg)|u_x|_{L^2}^2 + \frac{|u|_{\infty}^{2p-2}}{\underline{\varepsilon}}V. \label{ineq_H_to_use}
\end{align}
\end{lemma}
Now, for $E := V+H$, using \eqref{agmon_used}, we conclude that $|u|_{\infty}^2 \leq 6E$. 
Hence, combining  \eqref{ineq_V_fin2} and \eqref{ineq_H_to_use}, we obtain
\begin{equation}
\label{E_ineq_use}
\begin{aligned}
\dot{E} & \leq - \alpha (E)V-\Lambda (E)|u_x|_{L^2}^2
 \\ &
-\varepsilon(u(1))\Gamma(E)u(1)^4 -\varepsilon (u(0))\Gamma(E)u(0)^4,
\end{aligned}
\end{equation}
where
\begin{align*}
\alpha(E) :=&~ m_1+m_o-2(6E)^{(p-1)/2} - (6E)^{p-1}/\underline{\varepsilon},\\
\Lambda (E) :=&~ \underline{\varepsilon}-\sqrt{6E}\bar{\varepsilon}\big(\sqrt{6E}\big) -2(m_o+m_1) 
\\ &
- n \bigg(\sum_{i=1}^{n}\gamma_i^2(6E)^{i}\bigg) / (2\underline{\varepsilon}), \\
\Gamma(E) :=&~ \sum_{i=2}^{n}\frac{|\gamma_i|}{i+2}\bigg(M^{i-2}-(6E)^{(i-2)/2}\bigg).
\end{align*}
We will now exploit \eqref{E_ineq_use} to show that, if $E(0)$ is sufficiently small, then
\begin{align}
E(t) \leq E(0) \qquad \forall t\geq 0. \label{E_ineq}
\end{align}
To do so, we start noting that  
\begin{itemize}
\item When $\gamma_i\neq 0$ for some $i\in \{2,3,...,n\}$, then, based on  \eqref{omega_1}-\eqref{omega3}, we have 
\begin{align}
\min 
\left\{ \alpha(E(0)), \Lambda(E(0)), \Gamma(E(0)) \right\} >0. \label{init_condi}
\end{align}
\item When $\gamma_i=0$ for all 
$i \in \{2,3,...,n\}$, then 
\begin{align}
\min 
\left\{ \alpha(E(0)), \Lambda(E(0)) \right\} > 0. \label{init_condibis}
\end{align}
\end{itemize}
 Next, we show that, in the first case, 
\begin{align}
\left\{ \alpha(E(t)),\Lambda (E(t)), \Gamma(E(t)) \right\} > 0 \quad \forall t\geq 0. \label{to_contradict}
\end{align}
While, in the second case,
\begin{align}
\min \left\{ \alpha(E(t)),\Lambda (E(t)) \right\} > 0 \quad \forall t\geq 0. \label{to_contradictbis}
\end{align}
This would be enough to verify \eqref{E_ineq} according to \eqref{E_ineq_use}. 

We show \eqref{to_contradict} using contradiction. \eqref{to_contradictbis} can be proven in the same way. 

Suppose that \eqref{to_contradict} is false. Hence, by continuity of  $t\mapsto E(t)$, $E \mapsto \alpha(E)$, $E \mapsto \Lambda(E)$, and $E\mapsto \Gamma(E)$, we conclude the existence of $t' < +\infty$ such that 
\begin{align*} 
\min \left\{ \alpha(E(t')),  \Lambda (E(t')), \Gamma(E(t')) \right\} & = 0,  
\\
 \min \left\{ \alpha(E(t'')),\Lambda (E(t'')), \Gamma(E(t'')) \right\} & > 0 \quad \forall t'' \in [0,t').  
 \end{align*}
Now, since $\alpha$, $\Lambda$, and $\Gamma$ are nonincreasing, then there exists a time interval $(t_1,t_2)\subset [0,t')$ with $t_1<t_2$ such that $\dot{E}>0$ on $(t_1,t_2)$. However, according to \eqref{E_ineq_use}, we have $\dot{E}\leq 0$ on $(t_1,t_2)$, which yields a contradiction. 

Going back to \eqref{ineq_V_fin2}, using \eqref{E_ineq} and \eqref{omega_1}-\eqref{omega3}, we obtain
\begin{equation}
\label{V_to_integrate}
\begin{aligned}
\dot{V}  \leq & - \bigg(m_1+m_o-2(6E(0))^{(p-1)/2}\bigg)V  
\\ & 
- k_1u(1)^2-k_ou(0)^2. 
\end{aligned}
\end{equation}
By integrating \eqref{V_to_integrate} and noting that $E(0)=(1/2)\kappa_o(u_o)^2$, we conclude that \eqref{decay_L2} holds.

Using \eqref{agmon_used}, \eqref{decay_L2}, and the fact that $|u_x|_{L^2}^2\leq 2E \leq 2E(0) = \kappa_o(u_o)^2$, we also verify \eqref{decay_max}.

At this point, to verify \eqref{decay_H1}, we study the $H^1$ norm of $u$. Note that
\begin{align*}
\frac{d}{dt}\bigg( \exp^{(\sigma/3)t}H \bigg) = \frac{\sigma}{3}\exp^{(\sigma/3)t}H + \exp^{(\sigma/3)t}\dot{H}.
\end{align*}
Hence, using \eqref{ineq_H_to_use}, we have
\begin{align}
\frac{d}{dt}\bigg( \exp^{(\sigma/3)t}&H \bigg) \leq \frac{\sigma}{3}\exp^{(\sigma/3)t}H +\exp^{(\sigma/3)t} |u|_{\infty}^{2p-2}V/\underline{\varepsilon}  \nonumber \\
&~+ \exp^{(\sigma/3)t}\left(\frac{n}{2\underline{\varepsilon}}\right)\bigg(\sum_{i=1}^{n}\gamma_i^2|u|_{\infty}^{2i}\bigg)|u_x|_{L^2}^2. \nonumber
\end{align}
Furthermore, since $t\mapsto E(t)$ is non-increasing, then 
\begin{align}
|u|_{\infty}\leq \sqrt{6E(0)} = \sqrt{3}\kappa_o(u_o) \leq \sqrt{6\omega}. \label{to_bound_use}
\end{align}
Together with \eqref{decay_L2}, we obtain
\begin{equation}
\label{to_rewrite}
\begin{aligned}
\frac{d}{dt} & \bigg( \exp^{(\sigma/3)t} H \bigg) \leq \frac{\sigma}{3}\exp^{(\sigma/3)t}H  \\
&~+ \bigg(\sum_{i=1}^{n}3^i\gamma_i^2 \kappa_o^{2i}\bigg)\left(\frac{n}{2\underline{\varepsilon}}\right)\exp^{(\sigma/3)t}|u_x|_{L^2}^2  \\
&~+\frac{3^{p-1}\kappa_o^{2p-2}}{\underline{\varepsilon}}V(0)\exp^{-2(\sigma/3)t}. 
\end{aligned}
\end{equation}
Next, using \eqref{H_def}, inequality \eqref{to_rewrite} can be rewritten as 
\begin{align}
&\frac{d}{dt}\bigg( \exp^{\frac{\sigma}{3}t}H \bigg) \leq \frac{3^{p-1}\kappa_o^{2p-2}}{\underline{\varepsilon}}V(0)\exp^{-2\frac{\sigma}{3}t}+\frac{\sigma}{3}\exp^{\frac{\sigma}{3}t} \times  \nonumber
\\ & \bigg[ \frac{\lambda_1}{2}u(1)^2 + \frac{\mu}{4}u(1)^4+ \frac{\lambda_o}{2}u(0)^2+ \frac{\mu}{4}u(0)^4\bigg] 
\nonumber \\
&~+ \bigg[\frac{\sigma}{6} + \bigg(\sum_{i=1}^{n}3^i\gamma_i^2 \kappa_o^{2i}\bigg)\left(\frac{n}{2\underline{\varepsilon}}\right)\bigg] \exp^{(\sigma/3)t}|u_x|_{L^2}^2. \label{to_rewrite2}
\end{align} 
Now, using \eqref{decay_max}, the latter inequality becomes
\begin{equation}
\label{to_integrate_H}
\begin{aligned}
&\frac{d}{dt}\bigg( \exp^{(\sigma/3)t}H \bigg) \leq 
\frac{\sigma(\lambda_1+\lambda_o)}{3}|u_o|_{L^2} \kappa_o \exp^{-(\sigma/6)t}  
\\ & +\bigg[\frac{\sigma(\lambda_1+\lambda_o)}{6} + \frac{2\mu \sigma}{3}\kappa_o^2\bigg]|u_o|_{L^2}^2\exp^{-2(\sigma/3) t}  
\\ & +\frac{2\mu \sigma}{3}|u_o|_{L^2}^3\kappa_o \exp^{-7(\sigma/6)t} + \frac{\mu \sigma}{6}|u_o|_{L^2}^4\exp^{-5(\sigma/3) t} 
\\ & + \bigg[\frac{\sigma}{6}+\bigg(\sum_{i=1}^{n}3^i\gamma_i^2 \kappa_o^{2i}\bigg)\left(\frac{n}{2\underline{\varepsilon}}\right)\bigg] \exp^{(\sigma/3)t}|u_x|_{L^2}^2 
\\ & +\frac{3^{p-1}\kappa_o^{2p-2}}{\underline{\varepsilon}}V(0)\exp^{-2(\sigma/3)t}. 
\end{aligned}
\end{equation}
Next, we show that
\begin{align}
\int_{0}^{t}\exp^{(\sigma/3)s}|u_x(\cdot,s)|_{L^2}^2ds \leq \frac{3}{2\zeta}V(0). \label{L1H1}
\end{align}
To do so, we first note that
\begin{align}
\frac{d}{dt}\left( \exp^{(\sigma/3)t}V \right) =&~ \frac{\sigma}{3}\exp^{(\sigma/3)t}V + \exp^{(\sigma/3)t}\dot{V}. \label{to_combine_high1}
\end{align}
Since $t\mapsto E(t)$ is non-increasing, then \eqref{to_bound_use} holds. Consequently, using \eqref{omega3}, we obtain $\sum_{i=2}^{n}\frac{|\gamma_i|}{i+2}\big(M^{i-2}-|u|_{\infty}^{i-2}\big) \geq 0$. Hence,  according to \eqref{ineq_V_fin2}, we obtain
\begin{align}
\dot{V} \leq&~ -\sigma V - \zeta|u_x|_{L^2}^2, \label{to_combine_high2}
\end{align}
where $\zeta := \underline{\varepsilon} - \sqrt{3}\kappa_o\bar{\varepsilon}\left(\sqrt{3}\kappa_o\right)-2(m_o+m_1)$.

Combining \eqref{to_combine_high1} and \eqref{to_combine_high2}, and using \eqref{decay_L2}, we obtain
\begin{align*}
\frac{d}{dt} \left( \exp^{(\sigma/3)t}V \right) & \leq  \frac{\sigma}{3}\exp^{-2(\sigma/3)t}V(0)  -\exp^{(\sigma/3)t}\zeta|u_x|_{L^2}^2. 
\end{align*}
Integrating the latter between $0$ and $t\geq 0$, we obtain 
\begin{align*}
\exp^{(\sigma/3)t}V(t)\leq \frac{3}{2}V(0)-\zeta\int_{0}^{t}\exp^{(\sigma/3)s}|u_x(\cdot,s)|_{L^2}^2ds,
\end{align*}
which implies \eqref{L1H1}.

Finally,  integrating \eqref{to_integrate_H} from $0$ to $t\geq 0$ and using \eqref{L1H1}, we obtain 
\begin{align*}
&\exp^{(\sigma/3)t}H(t) \leq H(0)+ 2(\lambda_1+\lambda_o)|u_o|_{L^2} \kappa_o \nonumber \\
&~+\bigg(\frac{\lambda_1+\lambda_o}{4}+\mu\kappa_o^2\bigg)|u_o|_{L^2}^2+\frac{4\mu}{7}|u_o|_{L^2}^3\kappa_o +\frac{\mu}{10}|u_o|_{L^2}^4\nonumber \\
& + \bigg[ \frac{\sigma}{4 \zeta}+
\bigg(\sum_{i=1}^{n}3^i\gamma_i^2 \kappa_o^{2i}\bigg)\left(\frac{3n}{4 \zeta \underline{\varepsilon}}\right)
 +\frac{3^{p}\kappa_o^{2p-2}}{2\underline{\varepsilon}\sigma} \bigg] V(0). 
\end{align*}
Hence, \eqref{decay_H1} follows while noting that $\zeta \geq \underline{\varepsilon} - \sqrt{6\omega} \bar{\varepsilon}\left(\sqrt{6\omega}\right)-2(m_o+m_1)$ and $\sigma \geq m_1+m_o-2(6\omega)^{(p-1)/2}$.

To verify \eqref{max_ux}, we start by proving that 
\begin{align}
\lim_{t\to +\infty} \exp^{(\sigma/4)t}|u_{xx}(\cdot,t)|_{L^2}^2 = 0. \label{u_xx_int}
\end{align}
To do so, note that
\begin{align*}
\frac{d}{dt}\left(\exp^{(\sigma/4)t}H\right) = \frac{\sigma}{4}\exp^{(\sigma/4)t}H+\exp^{(\sigma/4)t}\dot{H}.
\end{align*}
In view of \eqref{final_lem2} in the Appendix and using Young's inequality, we obtain 
\begin{align*}
\dot{H} \leq -\frac{\underline{\varepsilon}}{4}|u_{xx}|_{L^2}^2+\frac{2}{\underline{\varepsilon}}|u|_{\infty}^{2(p-1)}V+\frac{n}{2\underline{\varepsilon}}\sum_{i=1}^{n}\gamma_i^2|u|_{\infty}^{2i}|u_x|_{L^2}^2.
\end{align*}
As a result, we obtain
\begin{align*}
\frac{d}{dt}&\left(\exp^{(\sigma/4)t}H\right) \leq \frac{\sigma}{4}\exp^{(\sigma/4)t}H+\exp^{\frac{\sigma}{4}t}\frac{2}{\underline{\varepsilon}}|u|_{\infty}^{2(p-1)}V \\
&~+\exp^{(\sigma/4)t}\frac{n}{2\underline{\varepsilon}}\sum_{i=1}^{n}\gamma_i^2|u|_{\infty}^{2i}|u_x|_{L^2}^2-\frac{\underline{\varepsilon}}{4}\exp^{(\sigma/4)t}|u_{xx}|_{L^2}^2.
\end{align*}
Integrating both sides of the latter inequality while using \eqref{decay_L2}, \eqref{decay_H1}, and \eqref{to_bound_use}, we get
\begin{equation*}
\begin{aligned}
\int_{0}^{+\infty} \exp^{(\sigma/4)t} & |u_{xx}(\cdot,t)|_{L^2}^2dt < +\infty,
\end{aligned}
\end{equation*}
which further implies \eqref{u_xx_int}.

Now, to establish 
\begin{align}
\lim_{t\to +\infty} \exp^{(\sigma/4)t}|u_t(\cdot,t)|_{L^2}^2 = 0, \label{int_ut_final}
\end{align}
we use \eqref{one} to obtain
\begin{align*}
|u_t|_{L^2}^2= \int_{0}^{1}\bigg(\varepsilon(u)u_{xx}+\sum_{i=1}^{n}\gamma_iu^iu_x+u^p\bigg)^2dx,
\end{align*}
which implies that 
\begin{equation}
\label{u_t_bound}
\begin{aligned}
|u_t|_{L^2}^2 & \leq 2 \max_{|s|\leq \sqrt{2}\kappa_o} \left\{ \varepsilon (s)^2 \right\} |u_{xx}|_{L^2}^2 
\\ & + 
4 \sum_{i=1}^{n}\gamma_i^2|u|_{\infty}^{2i}|u_x|_{L^2}^2   
+ 2|u|_{\infty}^{2(p-1)}|u|_{L^2}^2. 
\end{aligned}
\end{equation}
Hence, \eqref{int_ut_final} follows  from \eqref{u_t_bound}, \eqref{decay_L2}, \eqref{decay_max}, \eqref{decay_H1}, and \eqref{u_xx_int}.

Finally, to verify 
\begin{align*}
\lim_{t\to +\infty} \exp^{(\sigma/4)t}|u_x(\cdot,t)|_{\infty}^2 = 0, 
\end{align*}
it is enough to use Agmon's inequality, which entails 
\begin{align*}
|u_x|_{\infty}^2 \leq |u_x|_{L^2}^2 + 2|u_x|_{L^2}|u_{xx}|_{L^2},
\end{align*}
combined with \eqref{decay_H1} and \eqref{u_xx_int}.

\subsection{Proof of  Property \ref{prop2} (Well-posedness)} \label{well_pos_s}

Consider the following equation 
\begin{equation*}
\bar{\Sigma} \ : \ \left\lbrace
\begin{aligned}
&u_t = \tilde{\varepsilon} (u)u_{xx} + \bar{\gamma}(u) u_x + \bar{f}(u), \\
&u_x(0) = \bar{v}_o(u(0)), \ \ u_x(1) = \bar{v}_1(u(1)), \\
&u(x,0) = u_o(x).
\end{aligned}
\right.
\end{equation*}
The different terms in $\bar{\Sigma}$ are given by 
\begin{equation*}
\begin{aligned}
\tilde{\varepsilon}(u) &:= \Phi(u)\varepsilon(u)+1-\Phi(u), 
\\
\bar{\gamma}(u) &:= \Phi(u)\bigg(\sum_{i=1}^{n}\gamma_i u^i\bigg), \ \ \bar{f}(u) := \Phi(u) u^p, \\
\bar{v}_1(u) &:= \Phi(u) v_1(u), \ \ \bar{v}_o(u) := \Phi(u)v_o(u),
\end{aligned}
\end{equation*}
where $\Phi : \mathbb{R}\to [0,1]$ is of class $\mathcal{C}^2$ and verifies
\begin{equation*}
\Phi (u) := 
\left\lbrace 
\begin{aligned}
&1 \quad \text{if $|u|\leq \mathcal{M}$}, \\
&0 \quad \text{if $|u|\geq \mathcal{M}+1$},
\end{aligned}
\right. 
\end{equation*}
and 
$\mathcal{M} := \sqrt{4\sqrt{V(u_o)E(u_o)}+2V(u_o)} + \delta_{\mathcal{M}}$,
for some constant $\delta_{\mathcal{M}}>0$.

The solutions to $\bar{\Sigma}$ are defined analogously to the solutions to $\Sigma$; see Definition \ref{def1}.

We will now exploit the following lemma, proved in the Appendix, along with inequality \eqref{decay_max} to prove the well-posedness of $\Sigma$.
 
\begin{lemma}\label{bar}
Suppose that the compatibility condition \eqref{compatibility} holds. Then, Property \ref{prop2} is verified for $\bar{\Sigma}$. \hfill $\square$
\end{lemma}

A key step consists in showing that every complete solution $u$ to $\bar{\Sigma}$, under $\kappa_o(u_o) \leq \sqrt{2\omega}$, must verify  
\begin{align} \label{bar_M}
|u(x,t)| \leq \mathcal{M} \qquad \forall x\in [0,1], \qquad  
\forall t \geq 0. 
\end{align}
Hence, it would follow that
$$ \Phi(u(x,t))=1 \qquad \forall x \in [0,1], \qquad  \forall t\geq 0.$$
Leading to
\begin{equation} 
\label{all_terms}
\begin{aligned}
&\tilde{\varepsilon}(u)=\varepsilon(u), \ \ \bar{\gamma}(u) = \sum_{i=1}^{n}\gamma_i u^i, \ \ \bar{f}(u) = u^p, \\
&\bar{v}_1(u(1))=v_1(u(1)), \ \ \bar{v}_o(u(0)) = v_o(u(0)).
\end{aligned}
\end{equation}
The latter implies that $u$ is also a complete solution to $\Sigma$.
For uniqueness of complete solutions to $\Sigma$, we note that if $v$ is a different complete solution to $\Sigma$, then \eqref{bar_M} would hold, while replacing $u$ therein by $v$, since \eqref{decay_max} is verified on $[0,+\infty)$. Therefore, the identities in \eqref{all_terms}, while replacing $u$ therein by $v$,  are satisfied, and $v$ must also be a complete solution to $\bar{\Sigma}$. However, $\bar{\Sigma}$ admits a unique complete solution. Hence, $u = v$. 

Now, to verify \eqref{bar_M}, we distinguish between two situations. 

When $u_o = 0$, then $u \equiv 0$ is the unique complete complete solution to $\bar{\Sigma}$, and \eqref{bar_M} holds. 

When $u_o$ is not identically null, since $\delta_{\mathcal{M}}>0$, we use Agmon's inequality to obtain
\begin{align}
|u_o(x)|< \mathcal{M} \qquad \forall x\in [0,1]. \label{uo_m}
\end{align}
To find a contradiction,  we let $t'>0$ such that
\begin{align*}
|u(x,t')| & > \mathcal{M} \quad \text{for some} \quad  x \in [0,1], 
\\ 
|u(x,t'')| & \leq \mathcal{M} \quad \forall x \in [0,1] \quad \forall t'' \in [0,t'). 
\end{align*}
The existence of such a $t'$ is guaranteed by \eqref{uo_m} and the continuity of the map $t\mapsto |u(\cdot,t)|_{\infty}$. 

As a result, $u$ is a solution to $\Sigma$ on $[0,t')$ and \eqref{decay_max} holds for all $t \in [0,t')$ since $\kappa_o(u_o) \leq \sqrt{2\omega}$. This further leads to $|u(\cdot,t)|_{\infty} < \mathcal{M}$ for all $t\in [0,t')$. Finally, by continuity of the map $t\mapsto |u(\cdot,t)|_{\infty}$, we conclude that $|u(x,t')|\leq \mathcal{M} \quad \forall x\in [0,1]$, which yields a contradiction. 

\subsection{Proof of Property \ref{prop3} (Positivity)}

The proof uses the parabolic comparison principle in Lemma \ref{parabolic_principle}. Indeed, for a complete solution $u$ to $\Sigma$, we let, for all $x\in [0,1]$ and all $t\geq 0$, $
\varepsilon_p(x,t) := \varepsilon(u(x,t)), \ \ \gamma_p(x,t) := \sum_{i=1}^{n}\gamma_i u(x,t)^i$.

We note that $\gamma_p\in \mathcal{H}_{loc}([0,1]\times [0,+\infty))$ since $u \in \mathcal{H}^{2+\beta,1+\beta/2}_{loc}([0,1]\times [0,+\infty))$. Additionally,  $\varepsilon \in \mathcal{H}_{loc}(\mathbb{R})$ because $\varepsilon'\in \mathcal{H}_{loc}(\mathbb{R})$, and thus, 
$\varepsilon_p \in \mathcal{H}_{loc}([0,1]\times [0,+\infty))$.
Next, we note that 
$\varepsilon_{px} = u_x\varepsilon'(u)$. Thus,  $\varepsilon_{px}\in \mathcal{H}_{loc}([0,1]\times [ 0,+\infty))$ since $u \in \mathcal{H}^{2+\beta,1+\beta/2}_{loc}([0,1]\times [0,+\infty))$ and $\varepsilon'\in \mathcal{H}_{loc}(\mathbb{R})$. 
Finally, since the maps $u\mapsto u^p$, $u\mapsto -v_1(u)$ and $u\mapsto v_o(u)$ are nondecreasing on $[0,+\infty)$, applying Lemma \ref{parabolic_principle} under  $(f(u),g_1,g_o) := (u^p,-v_1,v_o)$, Property \ref{prop3} follows.

\subsection{Proof of Property \ref{prop4} (Invertibility)}

For all $l \in \{0,1\}$, when $\mu=0$, it follows that $u_x(l) = v_l(u(l)) = (-1)^l\lambda_l u(l)$ 
and  $u(l) = d_{l}(u_x(l)) :=  (-1)^{l} u_x(l)/\lambda_l$. Now, suppose that $\mu \neq 0$. In this case, according to \eqref{v1}-\eqref{v2}, we have 
\begin{align}
(-1)^{l+1}u(l)^3 + (-1)^{l+1}\frac{\lambda_l}{\mu}u(l) + \frac{1}{\lambda_l}u_x(l) = 0. \label{polynomial}
\end{align}
The latter is a cubic polynomial equation in $u(l)$, whose discriminant $\Delta := \frac{u_x(l)^2}{4\lambda_l^2} + \frac{\lambda_l^3}{27\mu^3}$ verifies $\Delta \geq 0$. Since the discriminant is nonnegative, then \eqref{polynomial} admits a unique real solution given by \eqref{cubic1}.

\section{Conclusion}
The focus in this paper was on stabilization of quasilinear parabolic equations that exhibit finite-time blow-up phenomena in open loop, with an estimate of the region of attraction. Specifically, we designed boundary controllers as cubic polynomials of boundary measurements, to guarantee exponential stability of the origin in the $L^2$ and $H^1$ norms with an estimate of the region of attraction, exponential convergence of the $H^2$ and $\mathcal{C}^1$ norms of the solutions to zero, well-posedness, and positivity of solutions. Moving forward, we aim to extend our framework to the case where $\varepsilon$ depends on $u_x$, which would lead to additional (possibly destabilizing) terms in the upperbounds on $\dot{V}$ and $\dot{H}$. It is unclear yet how we can adapt our approach to handle these terms. Another important problem is the extension to higher-dimensional domains. The main difficulty, in this case, is that Agmon's inequality, which is key in our stability analysis, does not hold in higher dimensions.

\appendix 

\section{Useful results} \label{appendixB}

In this section, we recall some results that are used in the paper. The following two inequalities are in \cite{burgers}.

\begin{lemma}[Agmon's inequality]
Given $u\in \mathcal{C}^1[0,1]$, we have $
|u|_{\infty}^2 \leq |u|_{L^2}^2+2|u|_{L^2}|u_x|_{L^2}$.
\end{lemma}

\begin{lemma}[Poincaré's inequality]\label{poincare}
Given $u\in \mathcal{C}^1[0,1]$, we have $
-2u(l)^2 \leq - |u|_{L^2}^2 + 4|u_x|_{L^2}^2$ for all $l\in \{0,1\}$.
\end{lemma}

The following result is in \cite[Theorem 7.4]{ladyzen}. 

\begin{lemma}\label{ladyz}
The system
\begin{equation*}
\bar{\Sigma} \ : \ \left\lbrace
\begin{aligned}
&u_t = \tilde{\varepsilon} (u)u_{xx} + \bar{\gamma}(u) u_x + \bar{f}(u), \\
&u_x(0) = \bar{v}_o(u(0)), \ \ u_x(1) = \bar{v}_1(u(1)), \\
&u(x,0) = u_o(x) \in \mathcal{H}^{2+\beta}[0,1], \ \ \beta \in (0,1),
\end{aligned}
\right.
\end{equation*}
admits a unique complete solution, and any other solution must be its restriction, if the following conditions hold.
\begin{itemize}
    \item[1-] $\exists$ $\underline{c},\bar{c},C>0$ such that, for all $u,u_x\in \mathbb{R}$,
\end{itemize}
    \begin{align}
    \underline{c} &\leq \tilde{\varepsilon}(u)\leq \bar{c}, \label{to_show_1} \\
    u\left( \bar{\gamma}(u)u_x +\bar{f}(u)\right) &\leq C\left( u_x^2+u^2+1\right), \label{to_show_2} \\
    u \bar{v}_1(u) &\leq C(u^2+1), \label{to_show_3}\\
    -u\bar{v}_o(u) &\leq C(u^2+1). \label{to_show_4}
    \end{align}
    \begin{itemize}
    \item[2-]  $\bar{v}_1,\bar{v}_o,\tilde{\varepsilon}, \bar{\gamma}, \bar{f}\in \mathcal{C}^2$.
    \item[3-]  $u_{o}'(1)=\bar{v}_1(u_o(1))$ and $u_{o}'(0)=\bar{v}_o(u_o(0))$.
\end{itemize}
\end{lemma}

The next lemma can be found in \cite[Condition (H$_1$)]{maximum_principle}.

\begin{lemma}[Parabolic comparison principle]\label{parabolic_principle}
Let the system 
\begin{equation*}
\Sigma_p \ : \ \left\lbrace
\begin{aligned}
&u_t = \varepsilon_p(x,t)u_{xx}+\gamma_p(x,t)u_x + f(u), \\
&u_x(0) = g_o(u(0)), \ \ u_x(1) = -g_1(u(1)), \\
&u(x,0) = u_o(x) \in \mathcal{H}^{2+\beta}[0,1], \ \ \beta \in (0,1),
\end{aligned}
\right.
\end{equation*}
where $\varepsilon_p, \varepsilon_{px}, \gamma_p\in \mathcal{H}_{loc}([0,1]\times [0,+\infty))$ and $f,g_1,g_o\in \mathcal{H}_{loc}(\mathbb{R})$. Moreover, suppose that, for any  $u_2 \geq u_1 \geq 0$, we have $f(u_2)-f(u_1) \geq 0$, 
$g_1(u_2)-g_1(u_1)\geq 0$, and 
$g_o(u_2)-g_o(u_1) \geq 0$. Finally, let $u_o(x)\geq 0$ for all $x\in [0,1]$. Then, any solution $u$ verifies $u(x,t) 
\geq 0$ for all $x\in [0,1]$ and for all $t$ in the interval of existence of $u$.
\end{lemma}
\section{Proof of lemmas}\label{appendixA}
\subsection{Proof of Lemma \ref{lem1}}\label{proof_1}
Differentiating $V$ with respect to time, we obtain 
$$\dot{V}
= \int_{0}^{1}u\big( \varepsilon (u)u_{xx}+\sum_{i=1}^{n}\gamma_i u^iu_x+u^p\big)dx. $$ Furthermore, using integration by parts, we get 
\begin{align*}
\int_{0}^{1}&u(x)\varepsilon (u(x))u_{xx}(x)dx \leq u(1)\varepsilon(u(1))v_1(u(1)) \\
&~- u(0)\varepsilon (u(0))v_o(u(0))- \big(\underline{\varepsilon}-|u|_{\infty}\bar{\varepsilon}(|u|_{\infty})\big)|u_x|_{L^2}^2.
\end{align*}
Using integration by parts one more time, we obtain 
\begin{align*}
\sum_{i=1}^{n}\gamma_i&\int_{0}^{1}u(x)^{i+1}u_x(x)dx= \sum_{i=1}^{n}\frac{\gamma_i}{i+2}\big(u(1)^{i+2}-u(0)^{i+2}\big).
\end{align*}
Finally, we note that $\int_{0}^{1}u(x)^{p+1}dx \leq 2|u|_{\infty}^{p-1}V$ to obtain 
\begin{align}
\dot{V} &\leq 2|u|_{\infty}^{p-1}V- \big(\underline{\varepsilon}-|u|_{\infty}\bar{\varepsilon}(|u|_{\infty})\big)|u_x|_{L^2}^2 \nonumber \\
&~+u(1)\varepsilon(u(1))v_1(u(1))-u(0)\varepsilon (u(0))v_o(u(0)) \nonumber \\
&~+\sum_{i=1}^{n}\frac{\gamma_i}{i+2}\big(u(1)^{i+2}-u(0)^{i+2}\big). \label{ineq_V_1}
\end{align}
The latter can be rewritten as
\begin{align}
\dot{V} &\leq 2|u|_{\infty}^{p-1}V- \big(\underline{\varepsilon}-|u|_{\infty}\bar{\varepsilon}(|u|_{\infty})\big)|u_x|_{L^2}^2+\sum_{i=1}^{n}\frac{\gamma_iu(1)^{i+2}}{i+2} \nonumber \\
&+u(1)\varepsilon(u(1))v_1(u(1))-u(0)\varepsilon (u(0))v_o(u(0)) \nonumber
\end{align}
\begin{align}
-\sum_{i=1}^{n}\frac{\gamma_i}{i+2}u(0)^{i+2} + \sum_{i=0}^{1}(-1)^i\sum_{k=0}^{1}(m_l+k_l)u(l)^2. \nonumber 
\end{align}
Next, by Poincaré's inequality in the Appendix, we have 
\begin{align*}
-m_lu(l)^2 \leq -m_lV+2m_l|u_x|_{L^2}^2 \qquad  \forall l\in \{0,1\}.
\end{align*}
On the other hand, if $i\geq 2$,  we have 
\begin{align*}
(-1)^{l+1}\gamma_i u(l)^{i+2}
&\leq u(l)\varepsilon (u(l)) \frac{|\gamma_i | |u|_{\infty}^{i-2}}{\underline{\varepsilon}}u(l)^3 \ \ \forall l\in \{0,1\}.
\end{align*}
Moreover, using Young's inequality, we get 
\begin{align*}
    \gamma_1 u(l)^3
    \leq u(l)\varepsilon (u(l))\bigg(\frac{|\gamma_1|}{2\underline{\varepsilon}}u(l)+\frac{|\gamma_1|}{2\underline{\varepsilon}}u(l)^3\bigg) \ \ \forall l\in \{0,1\}.
\end{align*}
Finally, note that
\begin{align*}
(m_l+k_l)u(l)^2 \leq u(l)\varepsilon (u(l)) \frac{m_l+k_l}{\underline{\varepsilon}}u(l) \ \ \forall l\in \{0,1\}.
\end{align*}
Combining the latter inequalities, we obtain the inequality in Lemma \ref{lem1}.
\subsection{Proof of Lemma \ref{lem2}}\label{proof_2}
Differentiating $H$ with respect to time, we obtain 
\begin{align}
\dot{H} =&~ \frac{1}{2}\frac{d}{dt}|u_x|_{L^2}^2 - v_1(u(1))u_t(1)+ v_o(u(0))u_t(0). \label{H_dot}
\end{align}

We now introduce the following claim. 
\begin{comment}
even if $u_{xt}$ is  not necessarily continuous, that it verifies 
In the particular case where $u_{xt}$ is continuous, we have $\frac{1}{2}\frac{d}{dt}|u_x|_{L^2}^2 = \int_{0}^{1}u_x(x)u_{xt}(x)dx$. Hence, using by-part integration, we obtain 
\begin{align}
\frac{1}{2}\frac{d}{dt}|u_x|_{L^2}^2
&= v_1(u(1))u_t(1)-v_o(u(0))u_t(0) \nonumber \\
&~- \int_{0}^{1}u_{xx}(x)u_t(x)dx. \label{H_dot_1}
\end{align}
Combining \eqref{H_dot} and \eqref{H_dot_1}, we find
\end{comment}
\begin{claim} \label{clm1}
It holds that $\dot{H} = - \int_{0}^{1} u_{xx}(x)u_t(x)dx$. 
\end{claim} 
Using Claim \ref{clm1}, we obtain 
\begin{align}
\dot{H} =&~ - \int_{0}^{1}u_{xx}\bigg(\varepsilon (u)u_{xx}+\sum_{i=1}^{n}\gamma_i u^iu_x+u^p\bigg)dx.\label{final_lem2}
\end{align}
Using Young's inequality, we verify
\begin{align*}
-\gamma_i \int_{0}^{1}u_{xx}u^iu_xdx \leq \frac{\underline{\varepsilon}}{2n}|u_{xx}|_{L^2}^2+ \frac{\gamma_i^2 n}{2\underline{\varepsilon}}|u|_{\infty}^{2i}|u_x|_{L^2}^2
\end{align*}
for all $i\in \{1,2,...,n\}$. Similarly, we have
\begin{align*}
- \int_{0}^{1}u_{xx}u^pdx \leq \frac{\underline{\varepsilon}}{2}|u_{xx}|_{L^2}^2+ \frac{|u|_{\infty}^{2p-2}}{\underline{\varepsilon}}V.
\end{align*}
Combining the latter inequalities with \eqref{final_lem2}, we obtain the inequality \eqref{ineq_H_to_use}.

\subsubsection{Proof of Claim \ref{clm1}}
In the particular case where $u_{xt}$ is continuous, we have 
$$\frac{1}{2}\frac{d}{dt}|u_x|_{L^2}^2 = \int_{0}^{1}u_x(x)u_{xt}(x)dx.$$
Hence, using integration by parts, we obtain 
\begin{align}
\frac{1}{2}\frac{d}{dt}|u_x|_{L^2}^2
&= v_1(u(1))u_t(1)-v_o(u(0))u_t(0) \nonumber \\
&~- \int_{0}^{1}u_{xx}(x)u_t(x)dx. \label{H_dot_1}
\end{align}
Combining \eqref{H_dot} and \eqref{H_dot_1}, we conclude on the identity in Claim \ref{clm1}.

Now, according to Definition \ref{def1}, $u_{xt}$ may not necessarily be continuous. Claim \ref{clm1} still holds in this case. One way to prove it is to consider, as suggested in \cite{ladyzen}, the map $(t,h)\mapsto H_h(t)$ defined for $(t,h)\in [0,+\infty)\times (0,+\infty)$ by
$$ H_h(t) := \frac{1}{2}\left|\frac{1}{h}\int_{t}^{t+h}u_x(\cdot,\tau)d\tau \right|_{L^2}^2. $$ 
Using l'Hôpital rule and the Lebesgue dominated convergence theorem, we first show that 
\begin{align}
\lim_{h\to 0^-} H_h(t) = \frac{1}{2}|u_x(\cdot,t)|_{L^2}^2 \quad \forall t\geq 0. \label{37}
\end{align}
Next, we evaluate $\dot{H}_h(t) := \frac{\partial }{\partial t}H_h(t)$ for all $t\geq 0$, where $\partial /\partial t$ denotes the partial derivative operator with respect to $t$. In particular, we have
\begin{align}
&\dot{H}_h(t) = \bigg[\bigg(\frac{\int_{t}^{t+h}u_x(x,\tau)d\tau}{h}\bigg) \bigg(\frac{u(x,t+h)-u(x,t)}{h}\bigg)\bigg]_{0}^{1} \nonumber \\
&~- \int_{0}^{1}\bigg(\frac{\int_{t}^{t+h}u_{xx}(x,\tau)d\tau}{h}\bigg) \bigg(\frac{u(x,t+h)-u(x,t)}{h}\bigg)dx. \nonumber
\end{align}
Integrating the latter identity with respect to $t$, from $0$ to any $t\geq 0$, and taking the limit as $h\to 0^-$ while using \eqref{37} and the Lebesgue dominated convergence theorem, we can show that
\begin{align}
\frac{1}{2}|u_x(\cdot,t)|_{L^2}^2 & = \int_{0}^{t}\big[u_x(x,s) u_t(x,s)\big]_{0}^{1}ds +\frac{1}{2}|u'|_{L^2}^2 \nonumber \\
&~- \int_{0}^{t}\int_{0}^{1}u_{xx}(x,s)u_t(x,s)dxds \quad \forall t\geq 0. \nonumber 
\end{align}
As a result, the map 
$t\mapsto |u_x(\cdot,t)|_{L^2}^2$ is differentiable on $[0,+\infty)$ and \eqref{H_dot_1} holds, which proves Claim \ref{clm1}.

\subsection{Proof of Lemma \ref{bar}}\label{proof_bar}
To prove the existence of a unique solution to $\bar{\Sigma}$, it is sufficient to verify the three conditions of Lemma \ref{ladyz}. 

We first consider the condition 1 from Lemma \ref{ladyz}. We need to show that there exist two constants $\underline{c},\bar{c}>0$ such that \eqref{to_show_1} holds. To prove the latter, we start by assuming that $u\in [-\mathcal{M},\mathcal{M}]$. Consequently, $\tilde{\varepsilon}(u)=\varepsilon(u)$, and we thus have, according to Assumption \ref{ass1}, $\underline{\varepsilon} \leq \tilde{\varepsilon}(u)$. Moreover, since $\varepsilon$ is continuous, then 
$$\tilde{\varepsilon}(u) \leq \max_{|s|\leq \mathcal{M}}\{\varepsilon(s) \}< +\infty.$$
Now, we suppose that $u\in (-\infty,-\mathcal{M}-1]\cup [\mathcal{M}+1,+\infty)$. In this case, we have $\tilde{\varepsilon}(u)=1$.

Finally, we suppose that $u\in (-\mathcal{M}-1,-\mathcal{M})\cup (\mathcal{M},\mathcal{M}+1)$. If $\Phi(u)\leq 1/2$, then 
\begin{align*}
\frac{1}{2}\leq \tilde{\varepsilon}(u) \leq \max_{\mathcal{M}\leq |s|\leq \mathcal{M}+1}\{\varepsilon(s)\}+1 < +\infty.
\end{align*}
Otherwise, if $\Phi(u)>1/2$, then 
\begin{align*}
\frac{1}{2}\underline{\varepsilon} \leq \tilde{\varepsilon}(u) \leq \max_{\mathcal{M}\leq |s|\leq \mathcal{M}+1}\{\varepsilon(s)\}+1. 
\end{align*}
We conclude that \eqref{to_show_1} holds with 
\begin{align*}
\underline{c} &:= \min\left\{\frac{1}{2}\underline{\varepsilon},1\right\}, \\
\bar{c} &:= \max\left\{\max_{\mathcal{M}\leq |s|\leq \mathcal{M}+1}\{\varepsilon(s)\}+1,\max_{|s|\leq \mathcal{M}}\{\varepsilon(s)\}\right\}.
\end{align*}

Next, we prove \eqref{to_show_2}. To do so, we first apply Young's inequality to obtain, for all $u_x,u\in \mathbb{R}$, $u\left( \bar{\gamma}(u)u_x+\bar{f}(u)\right) \leq \frac{1}{2}u^2\left(\bar{\gamma}(u)^2+\bar{f}(u)^2\right)+\frac{1}{2}u_x^2$. Moreover, note that we have $\bar{f}(u)^2 = \Phi(u)^2 u^{2p} \leq (\mathcal{M}+1)^{2p}$. Similarly, note that
\begin{align*}
\bar{\gamma}(u)^2 = \Phi(u)^2\left(\sum_{i=1}^{n}\gamma_i u^i\right)^2 \leq \left(\sum_{i=1}^{n}|\gamma_i|\right)^2(\mathcal{M}+1)^2,
\end{align*}
which allows us to conclude that \eqref{to_show_2} holds with 
\begin{align*}
C\geq \frac{1}{2}\max\left\{1,(\mathcal{M}+1)^{2p},\left(\sum_{i=1}^{n}|\gamma_i|\right)^2(\mathcal{M}+1)^2\right\}.
\end{align*}

Now, to establish \eqref{to_show_3}, we apply Young's inequality to obtain 
\begin{align*}
u\bar{v}_1(u) &\leq \frac{1}{2}u^2+\frac{1}{2}\Phi(u)^2\left(\lambda_1u+\mu u^3\right)^2 \\
&\leq \frac{1}{2}u^2+\frac{1}{2}\left(\lambda_1(\mathcal{M}+1)+\mu (\mathcal{M}+1)^3\right)^2.
\end{align*}
As a consequence, \eqref{to_show_3} is verified with 
\begin{align*}
C \geq \frac{1}{2}\max\left\{1,\left(\lambda_1(\mathcal{M}+1)+\mu (\mathcal{M}+1)^3\right)^2\right\}.
\end{align*}
We prove \eqref{to_show_4} in a similar way.

The condition 2 from Lemma \ref{ladyz} follows from the fact that $u\mapsto v_1(u),v_o(u),\varepsilon(u),u^i$ for $i\in \{1,2,...,n,p\}$, and $\Phi : \mathbb{R}\to [0,1]$, are of class $\mathcal{C}^2$.

Finally, to establish the condition 3 from Lemma \ref{ladyz}, it is enough to note that, according to Agmon's inequality, we have 
$|u_o(x)|< \mathcal{M} \quad  \forall x\in [0,1]$.

Consequently, $\bar{v}_1(u_o(1))=v_1(u_o(1))=u_o'(1)$ and $\bar{v}_o(u_o(0))=v_o(u_o(0))=u_o'(0)$, which concludes the proof.

\par\noindent 
\parbox[t]{\linewidth}{
\noindent\parpic{\includegraphics[height=1.5in,width=1in,clip,keepaspectratio]{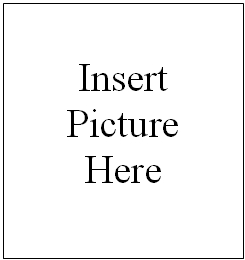}}
\noindent {\bf Mohamed Camil Belhadjoudja}\
received a Control-Engineer and MS degrees from the National Polytechnic School, Algeria, in 2021, and a MS degree in Control from the Université Grenoble Alpes, France, in 2022. He is a Ph.D. student at GIPSA-lab, France. His research is focused on control and state estimation for nonlinear PDEs.}

%\vspace{4\baselineskip}

\par\noindent 
\parbox[t]{\linewidth}{
\noindent\parpic{\includegraphics[height=2in,width=1in,clip,keepaspectratio]{insert-picture-here.jpg}}
\noindent {\bf Mohamed Maghenem}\
 received his Control-Engineer degree from the Polytechnical School of Algiers, Algeria, in 2013, his  PhD degree on Automatic Control from the University of Paris-Saclay, France, in 2017.  He was a Postdoctoral Fellow at the Electrical and Computer Engineering Department at the University of California at Santa Cruz from 2018 through 2021. M. Maghenem holds a research position at the French National Centre of Scientific Research (CNRS) since January 2021.  His research interests include control systems theory (linear,  non-linear,  and hybrid) to ensure (stability, safety, reachability, synchronisation, and robustness); with applications to mechanical systems,  power systems,  cyber-physical systems,  and some partial differential equations.}
 
 % \vspace{4\baselineskip}

\par\noindent 
\parbox[t]{\linewidth}{
\noindent\parpic{\includegraphics[height=1.5in,width=1in,clip,keepaspectratio]{insert-picture-here.jpg}}
\noindent {\bf Emmanuel Witrant}\
obtained a B.Sc. in Aerospace Eng. from Georgia Tech in 2001 and a Ph.D. in Automatic Control from Grenoble University in 2005. He joined the Physics department of Univ. Grenoble Alpes and GIPSA-lab as an Associate Professor in 2007 and became Professor in 2020. His research interest is focused on finding new methods for modeling and control of inhomogeneous transport phenomena (of information, energy, gases), with real-time and/or optimization constraints. Such methods provide new results for automatic control (time-delay and distributed systems), controlled thermonuclear fusion (temperature/magnetic flux profiles in tokamak plasma), environmental sciences (atmospheric history of trace gas from ice cores measurements) and Poiseuille’s flows (ventilation control in mines, car engines and intelligent buildings).}
% \vspace{4\baselineskip}

\par\noindent 
\parbox[t]{\linewidth}{
\noindent\parpic{\includegraphics[height=1.5in,width=1in,clip,keepaspectratio]{insert-picture-here.jpg}}
\noindent {\bf Miroslav Krstic}\
is Distinguished Professor at UC San
Diego. He is Fellow of IEEE, IFAC, ASME, SIAM, AAAS,
IET, AIAA (Assoc. Fellow), and Serbian Academy of
Sciences and Arts. He has received the Bellman Award,
Bode Lecture Prize, Reid Prize, Oldenburger Medal,
Nyquist Lecture Prize, Paynter Award, Ragazzini Award,
IFAC Nonlinear Control Systems Award, Distributed
Parameter Systems Award, Adaptive and Learning Systems Award, Chestnut Award, AV Balakrishnan Award,
CSS Distinguished Member Award, the PECASE, NSF
Career, and ONR YIP, and the Schuck and Axelby paper
prizes. He serves as Editor-in-Chief of Systems \& Control Letters and Senior Editor in Automatica.}
\end{document}